\documentclass[12pt,twoside]{amsart}
\usepackage{amssymb,amsmath,amsthm, amscd, enumerate, mathrsfs, upgreek}
\usepackage{graphicx, hhline, tikz}
\usepackage[colorlinks=true,pagebackref,hyperindex]{hyperref}
\usepackage[all]{xy}

\usepackage{url}

\usepackage{color}
\usepackage[backrefs, alphabetic, initials]{amsrefs}
\usepackage{color}  
\usepackage{latexsym, bm}
\usepackage[T1]{fontenc} 
\usepackage{fancyhdr}
\usepackage{enumerate}
\title[Strictly nef divisors]
{Strictly nef divisors on K-trivial fourfolds}

\author{Haidong Liu} 
\author{Shin-ichi Matsumura}

\date{\today, version 0.05}

\address{Peking University, Beijing International Center for Mathematical Research, 
Beijing 100871, China}
\email{hdliu@bicmr.pku.edu.cn, jiuguiaqi@gmail.com}

\address{Mathematical Institute, Tohoku University, 6-3, Aramaki Aza-Aoba, Aoba-ku, Sendai 980-8578, Japan}
\email{mshinichi@m.tohoku.ac.jp, mshinichi0@gmail.com}
\urladdr{\url{https://sites.google.com/view/math-matsumura/}}

\DeclareMathOperator{\codim}{codim}

\DeclareMathOperator{\reg}{reg}

\DeclareMathOperator{\strutt}{\mathcal{O}}

\newcommand{\e}{\varepsilon}

\newcommand{\bX}{\mathcal{X}}
\newcommand{\bL}{\mathcal{L}}

\newcommand{\Def}[1]{{\rm{Def}}(#1)}
\newtheorem{thm}{Theorem}[section]
\newtheorem{lem}[thm]{Lemma}
\newtheorem{prop}[thm]{Proposition}
\newtheorem{conj}[thm]{Conjecture}
\newtheorem{cor}[thm]{Corollary}

\theoremstyle{definition}

\newtheorem{defn}[thm]{Definition}
\newtheorem{rem}[thm]{Remark}
\newtheorem*{ack}{Acknowledgments}

\newtheorem{case}{Case}

\makeatletter
    
    \@addtoreset{equation}{section}
\makeatother
\begin{document}

\begin{abstract}
In this paper, we prove the ampleness conjecture and Serrano's conjecture 
for strictly nef divisors on K-trivial fourfolds. 
Specifically, we show that any strictly nef divisors 
on projective fourfolds with trivial canonical bundle and vanishing irregularity are ample. 
\end{abstract}

\maketitle 

\tableofcontents

\section{Introduction}\label{sec1}

This paper is devoted to the study of strictly nef divisors
on projective manifolds with trivial canonical bundle, 
motivated by the following so-called  \textit{ampleness conjecture}: 

\begin{conj}[Ampleness conjecture]\label{conj.amp}
Let $X$ be a projective manifold with trivial canonical bundle. Then, any strictly nef divisors on $X$ are ample.
\end{conj}

Here, a $\mathbb Q$-Cartier divisor 
$L$ on a projective variety $X$ is said to be \textit{strictly nef} 
if the intersection number $L \cdot C$ is positive for any curve $C \subset X$.
The ampleness conjecture follows 
from Serrano's conjecture (see \cites{ccp, serrano} and Conjecture \ref{serrano.log.version}) or 
the semiampleness conjecture (see \cites{kollar, lop1, lp20, liu-svaldi, verbitsky}). 
Hence, an affirmative solution of the ampleness conjecture provides evidence in support of these conjectures. 
Furthermore, the ampleness conjecture does not hold 
without assuming that $X$ has the trivial canonical bundle 
(for example, see Mumford and Ramanujam's examples in ~\cite{hart-1}*{appendix to Chapter I} 
or  Campana and Okawa's examples in \cite{okawa}); 
therefore, it is also interesting in terms of studying  
characteristics of projective manifolds with trivial canonical bundle.

The Beauville--Bogomolov--Yau decomposition \cite{beauville} shows that 
any projective manifold with trivial canonical bundle can be decomposed, 
up to a finite \'etale cover, 
into the product of abelian varieties, (strict) Calabi--Yau manifolds, and (simple) hyperk\"ahler  manifolds. 
In this paper,  a simply connected projective manifold $X$ with trivial canonical bundle 
is called a \textit{Calabi--Yau manifold} (resp. a \textit{hyperk\"ahler manifold})  
if $h^i(X, \mathcal{O}_X)=0$ for $0<i<\dim X$ (resp. $H^2(X, \mathcal O_X) \cong \mathbb C$). 
The ampleness conjecture holds for abelian  varieties (for example, see \cite{bl}), 
and both strict nefness and ampleness are preserved under finite coverings; 
hence, the projective manifold $X$ in Conjecture \ref{conj.amp} 
can be assumed to be a Calabi--Yau manifold or a  hyperk\"ahler manifold.

The ampleness conjecture has been completely solved in dimension two,
and in dimension three except for Calabi--Yau threefolds  satisfying  $c_3(X) = 0$ or $\nu(L)  = 1$ 
(see \cites{liu-svaldi, oguiso, serrano} and the references therein); 
however, the remaining case in dimension three is still open. 
In this paper, we leave the three-dimensional case behind and turn to K-trivial fourfolds. 
Throughout this paper, we use {\textbf{the term of a K-trivial manifold 
to denote a projective manifold $X$  with $K_X \sim 0$ and $h^{1}(X, \mathcal{O}_{X})=0$}}. 
The following result, which is one of the main results in this paper, 
solves Conjecture \ref{conj.amp} for K-trivial fourfolds.

\begin{thm}\label{thm-main}
Any strictly nef $\mathbb Q$-Cartier divisor $L$ 
on a K-trivial fourfold $X$ is ample.
\end{thm}
Note that the three-dimensional case cannot be reduced to the four-dimensional case, 
because the condition $h^{1}(X, \mathcal{O}_{X})=0$, 
which is also called vanishing irregularity, excludes 
the product of a Calabi--Yau manifold and an abelian variety from K-trivial manifolds. 

The proof of Theorem \ref{thm-main} is divided into two parts: Theorem \ref{thm.2} and \ref{thm.1}.
Theorem \ref{thm.2} solves the generalized non-vanishing conjecture (for example, see \cite{lp20}) 
and Theorem \ref{thm.1}  solves the abundance conjecture for 
strictly nef and effective divisors on K-trivial fourfolds.

\begin{thm}[Theorems \ref{thm.num3} and \ref{thm-num=2}]\label{thm.2}
If $L$ is a strictly nef $\mathbb Q$-Cartier divisor on a  K-trivial fourfold $X$, 
then we have $\kappa(L) \geq 0$ $($i.e., $L$ is $\mathbb Q$-effective$)$.  
\end{thm}

The proof of Theorem \ref{thm.2}, which is discussed in Section \ref{sec4}, 
depends on detailed studies for cases according to the numerical dimension $\nu(L)$. 
Based on the discussion of~\cites{lop, lp, liu-svaldi, oguiso, wilson, wilson1}, 
we will construct a section of a positive multiple of $L$ 
by applying the Kawamata--Viehweg vanishing theorem, 
the hard Lefschetz theorem, 
Hirzebruch--Riemann--Roch formula, 
and the Bedford--Taylor product of positive currents.

\begin{thm}[Theorem \ref{thm.abundance}]\label{thm.1}
Let $L$ be a strictly nef $\mathbb Q$-Cartier divisor on  a K-trivial fourfold $X$.
If we have $\kappa(L)\geq 0$, then $L$ is ample.
\end{thm}

Theorem \ref{thm.1} follows from the abundance theorem, 
which holds for projective varieties of dimension $\leq 3$; 
however, the abundance conjecture is still open, even in the case of K-trivial fourfolds. 
Therefore, we need to get around this difficulty by fully using the strict nefness.

For the proof of Theorem \ref{thm.1}, we focus on Serrano's  conjecture. 
Serrano's original conjecture can be viewed 
as a weak analogue of Fujita's conjecture (see ~\cite{laz}*{\S~10.4.A}) for strictly nef divisors; 
hence, as in Fujita's conjecture, 
it is natural to generalize Serrano's conjecture to log pairs. 
In this paper, we propose a log version of Serrano's conjecture (see Conjecture \ref{serrano.log.version}). 
The log version naturally appears in considering the inductive proof of Serrano's conjecture.

\begin{conj}[Log version of Serrano's conjecture] \label{serrano.log.version}
Let $(X,\Delta)$ be a log canonical pair of dimension $n$ and $L$ be a strictly nef Cartier divisor on $X$.
Then, the $\mathbb Q$-Cartier divisor $K_X+\Delta+tL$ is ample for $t>2n$.
\end{conj}

Serrano's conjecture has not been fully solved even for K-trivial threefolds. 
Nevertheless, we can partially solve the log version of Serrano's conjecture (see Theorem \ref{thm.3}), 
which is one of the main results of this paper.

\begin{thm}[Theorems \ref{thm.abundance.eff}, \ref{thm.ab.dim3} and Corollary \ref{cor.abundance.eff}]
\label{thm.3}
Let $X$ be a K-trivial manifold of dimension $n\leq 4$, 
$\Delta$ be a nonzero effective divisor, and $L$ be a strictly nef  divisor on $X$. 
Then, the Cartier divisor $\Delta+tL$ is ample for $t\gg 1$.
\end{thm}

Theorem \ref{thm.1} is a direct corollary of Theorem \ref{thm.3}. 
The key point here is that $\Delta$  in Theorem \ref{thm.3} is nonzero; 
otherwise, the statement of Theorem \ref{thm.3} coincides with 
Serrano's original conjecture, which is still open even in dimension three.

\begin{ack}
The authors would like to thank Prof. Vladimir Lazi\'c 
for reading a draft of this paper carefully and pointing out several mistakes. 
The discussions with him significantly improved this paper. 
The first author is very grateful to Prof. Chen Jiang for his useful conversations and for reading several drafts of this paper. 
He would like to thank Professors Osamu Fujino, Zhengyu Hu, 
Roberto Svaldi, and Zhiyu Tian 
for useful discussions and suggestions,
and to thank Prof. Wenfei Liu for pointing out that \cite{han-liu}*{Question 3.5} is stronger than Conjecture \ref{serrano.log.version}.
He also would like to thank Prof. Thomas Peternell for pointing out a disastrous gap in a draft of this paper.
The first author is supported by the NSFC (grant no.~12001018).
The second author was supported 
by Grant-in-Aid for Scientific Research (B) $\sharp$ 21H00976 
and Fostering Joint International Research (A) $\sharp$19KK0342 from JSPS. 
\end{ack}

Throughout this paper, 
a {\em scheme} is always assumed to be separated and of finite type over $\mathbb{C}$, 
and a {\em variety} is a reduced and irreducible algebraic scheme. 
We will freely use the basic notation of the minimal model program 
in \cites{fujino-foundations, kollar-mori, laz}.

\section{Preliminaries}\label{sec2}
In this section, we present preliminary results.

\subsection{Almost strictly nef divisors}

The strict nefness is not preserved under birational modifications, 
as in the inductive proof of Serrano's conjecture in \cites{ccp, serrano}.
Therefore, following \cite{ccp}, 
it would be more natural and useful 
to define a birational analogue of the strict nefness,
which is called the almost strict nefness. 

\begin{defn}[\cite{ccp}*{Definition 1.1}]\label{defn.asn}
A $\mathbb Q$-Cartier divisor $L$ on a normal variety $X$ 
is called {\em{almost strictly nef}}
if there exists a surjective birational 
morphism $\mu\colon X\to X^*$ and a strictly nef $\mathbb Q$-Cartier divisor $L^*$ on $X^*$
such that $L=\mu^*L^*$. 
When a birational morphism $\mu\colon X\to X^*$  is specified, 
we say that $L$ is {\em{almost strictly nef with respect to $\mu$}}. 
\end{defn}

We will use the following facts without any mention.

\begin{prop}\label{prop.asn}
Let $L$ be a $\mathbb Q$-Cartier divisor on a normal projective variety $X$.
\begin{itemize}
\item[(1)] Let $f\colon Y\to X$ be a finite morphism. 
Then, the pull-back $f^*L$ is strictly nef on $Y$ if and only if $L$ is strictly nef on $X$.
\item[(2)] If $f\colon Y\to X$ is a generally finite morphism and $L$ is almost strictly nef, 
then $f^*L$ is almost strictly nef on $Y$.
\item[(3)] Let $L$ be an almost strictly nef $\mathbb Q$-Cartier divisor
with respect to $\mu\colon X\to X^*$. 
If a normal subvariety $Z\subset X$ is not contained in the $\mu$-exceptional locus, 
then $L|_Z$ is almost strictly nef on $Z$.
\end{itemize}
\end{prop}

\begin{rem}\label{rem.asn}
Let $\mu\colon X\to X^*$  be the birational 
morphism in Definition \ref{defn.asn} with a strictly nef divisor $L^*$ on $X^*$,
and $X\to X'\to X^*$ be the Stein factorization of $\mu$.
Then, since the morphism $\pi\colon X'\to X^*$ is finite, 
the pull-back $\pi^*L^*$  is also strictly nef by Proposition \ref{prop.asn} (1).
Hence, by replacing $X^*$ with $X'$ and $L^*$ with $\pi^*L^*$, 
we can always assume that $X^*$ is normal and
$\mu\colon X\to X^*$  has connected fibers.
\end{rem}

\begin{lem}\label{nef-lem}
Let $(X, \Delta)$ be a log canonical pair of dimension $n$ and 
$L$ be an almost strictly nef Cartier divisor with respect to $\mu\colon X\to X^*$.
If $K_X+\Delta$ is $\mu$-nef,
then the $\mathbb Q$-Cartier divisor $K_{X}+\Delta+tL$ is nef for $t\geq 2n$. 
In particular, if $L$ is a strictly nef Cartier divisor on $X$, 
then  $K_{X}+\Delta+tL$ is nef for $t\geq 2n$. 
\end{lem}

\begin{proof}
The proof follows from the cone theorem for log canonical pairs (see \cite{fujino}*{Theorem 1.4}), 
which states that  any curve $C$ on $X$ is numerically equivalent to a linear combination
\[
r_1F_1+\cdots+r_kF_k+M,
\]
where $r_i$ is a positive real number, 
$F_i$ is a rational curve on $X$
with $0<-(K_{X}+\Delta)\cdot F_i\leq 2n$, 
and $M$ is the limit of effective real $1$-cycles with $(K_{X}+\Delta)\cdot M\geq 0$.
Let $L^*$ be a strictly nef $\mathbb Q$-Cartier divisor on $X^*$
such that $L=\mu^*L^*$.
Since $K_{X}+\Delta$ is $\mu$-nef, 
the rational curve $F_i$ is not contracted by $\mu$ for any $i$.  
In particular, we can see that 
\[
L\cdot F_i\geq L^*\cdot \mu(F_i)>0,
\]
which implies that $L\cdot F_i\geq 1$ since $L$ is Cartier.
Then, from $-(K_{X}+\Delta)\cdot F_i\leq 2n$, 
we obtain $(K_{X}+\Delta+2nL)\cdot F_i\geq 0$ for any $i$. 
Then, since $L$ is nef, we can see that 
\[
(K_{X}+\Delta+tL)\cdot C\geq (K_{X}+\Delta+2nL)\cdot C\geq (K_{X}+\Delta+2nL)\cdot M\geq 0
\] 
for $t\geq 2n$.
\end{proof}

Under the same situation as in Lemma \ref{nef-lem}, 
it is conjectured that a $(K_{X}+\Delta)$-negative extremal ray $F$ satisfies 
that $0<-(K_{X}+\Delta)\cdot F\leq n+1$ (see \cite{fujino-foundations}*{Remark 4.6.3}), 
which predicts the sharper bound $t\geq n+1$  in Lemma \ref{nef-lem}. 
Based on this observation, 
a more precise statement than that of Conjecture \ref{serrano.log.version} 
was formulated in \cite{han-liu}*{Question 3.5}. 
However, this sharper  bound  seems far from reaching,
so  Conjecture \ref{serrano.log.version} is more workable in practice.
The sharper bound in Lemma \ref{nef-lem} can be obtained 
when $X$ is smooth and $\Delta=0$.

\begin{lem}\label{lem.sm.nef}
Let $X$ be a projective manifold of dimension $n$ and 
$L$ be an almost strictly nef Cartier divisor with respect to $\mu$.
Assume that $K_X$ is $\mu$-nef. 
Then, the Cartier divisor $K_{X}+tL$ is nef
for $t\geq n+1$.
\end{lem}
\begin{proof}
By the smooth version of the cone theorem, 
any curve  $F_i$ representing the $K_X$-negative extremal ray
satisfies that $0<-K_{X}\cdot F_i\leq n+1$, 
which finishes the proof. 
\end{proof}

For almost strictly nef  divisors, we have the following conjecture
(see \cite{ccp}*{Conjecture 2.2}), which is closely related to
the log version of Serrano's conjecture:
\begin{conj}\label{big-conj}
Let $(X,\Delta)$ be a log canonical pair of dimension $n$ and $L$ be an almost strictly nef Cartier divisor on $X$.
Then, the $\mathbb Q$-Cartier divisor  $K_X+\Delta+tL$ is big for $t>2n$. 
\end{conj}
\subsection{Hirzebruch--Riemann--Roch formula}
Let $L$ be a Cartier divisor on a K-trivial fourfold $X$. 
Recall that our definition of K-trivial manifolds requires $h^1(X,\mathcal O_X)=0$. 
Then, since we have 
$$h^3(X, \mathcal O_X)=h^1(X,\mathcal O_X(K_X))=h^1(X,\mathcal O_X)=0,$$ 
we obtain 
\begin{equation}\label{eq.hrr.1}
  \begin{split}
&\chi(X, mL)=\frac{L^4}{24}m^4+\frac{c_2(X) \cdot L^2}{24}m^2+\frac{3c^2_2(X)-c_4(X)}{720},   \\
&\chi(\mathcal O_X)=\frac{3c^2_2(X)-c_4(X)}{720}
=\sum^4_{i=0} (-1)^ih^i(\mathcal O_X)=2+h^2(\mathcal O_X)\geq 2
  \end{split}
\end{equation}
from  the Hirzebruch--Riemann--Roch formula. 
It follows that $c_2(X)$ is pseudoeffective from Miyaoka's result \cite{miyaoka}*{Theorem 6.6}; 
hence we have $c_2(X) \cdot L^2\geq 0$  when $L$ is nef.
Furthermore, the term of degree $0$ in the expansion of $\chi(X, mL)$ is equals to $\chi(\mathcal O_X) (\geq 2)$. 
These observations yield the following lemma. 
\begin{lem}\label{lem.fc}
Let $L$ be a nef  Cartier divisor 
on a K-trivial fourfold $X$. 
Then, the inequality $\chi(X, mL)\geq 2$ holds for all $m\geq 0$.
\end{lem}
\subsection{Relative terminal and canonical models}
Let $X$ be a normal variety of dimension $n$ and $g\colon Z\to X$ be a resolution of singularities of $X$.
By applying the result in \cite{bchm}, we can run a relative $K_Z$-MMP over $X$; 
then, we obtain a birational morphism $f \colon Y\to X$ 
such that $K_{Y}$ is $f$-nef and that $Y$ is a $\mathbb Q$-factorial terminal variety. 
Replacing $Y$ by its relative canonical model $Y^c$,
we further obtain $f^c \colon Y^c\to X$ such that
$K_{Y^c}$ is $f^c$-ample. 
The morphism $f\colon Y\to X$ (resp. $f^c \colon Y^c\to X$)
is called a {\em{relative terminal model}} (resp. {\em{relative canonical model}}) of $X$.
In dimension two, the relative terminal model is the relative minimal resolution; 
in particular, it is uniquely determined by $X$.
In the higher dimensional case, 
the relative terminal and canonical models depend on the initial choice of a resolution 
$g\colon Z\to X$.

\subsection{Prime Calabi--Yau divisors}
Following the inductive strategy of \cites{ccp, serrano}, 
we consider Conjecture \ref{serrano.log.version} 
(i.e., the log version of Serrano's conjecture) 
for some divisor of a K-trivial fourfold $X$.
As stated in the introduction, however,
the remaining Calabi--Yau threefolds of Serrano's conjecture 
are far from reaching.
Therefore, in considering  Conjecture \ref{serrano.log.version} for K-trivial fourfolds, 
we need a careful discussion for special prime divisors of Calabi--Yau type. 
In this paper, 
a prime divisor $D$ on a normal variety $X$ is said to be a \emph{prime Calabi--Yau divisor}  if
$D$ is a normal variety with at worst canonical singularities 
satisfying that $K_D\sim 0$ and  $H^1(D, \mathcal{O}_D)=0$. 
The following proposition shows that 
a prime Calabi--Yau divisor $D$ on a  K-trivial manifold $X$ 
induces a fibration $\pi\colon X\to C$ onto a smooth curve $C$.

\begin{prop}\label{prop.cy}
Let $X$ be a K-trivial manifold. 
Then, any prime Calabi--Yau divisor $D$ on $X$ is a semiample divisor with $\nu(D)=\kappa(D)=1$.
\end{prop}
\begin{proof}

Let $C$ be a curve on $X$. If $C\not\subset D$, then we have $D\cdot C\geq 0$. 
If $C\subset D$, then we have $D\cdot C=D|_D\cdot C=K_D\cdot C=0$ from $K_D\sim 0$.
Hence, we can see that $D$ is a nef divisor on $X$. 
Let $H_1, H_2, \cdots , H_{n-2}$  
be very ample and general divisors on $X$, where $n:=\dim X$. 
Then, we have  
\[
D^2\cdot H_1\cdot H_2\cdots H_{n-2}=K_D\cdot H_1|_D\cdot H_2|_D \cdots H_{n-2}|_D = 0, 
\]
and so  $D^2 $ is numerically trivial on $X$, 
which indicates that $\nu(D)=1$ by definition. 

We now  consider the long exact sequence induced by the following short exact sequence
\[
0\to  \mathcal O_X \to  
\mathcal O_X(D)\to \strutt_{D}(D)\simeq \strutt_{D}\to 0.
\]
Then, we obtain that 
$$h^0(X, D)=h^0(D,\strutt_{D})+h^0(X,\strutt_{X})=2$$ 
from $H^1(X, \mathcal O_X)=0$,  
which indicates that 
$h^0(X, kD)\geq h^0(X, D)=2$; 
in particular, we have $\kappa (D)\geq 1$.
Hence, we obtain $1=\nu(D)\geq \kappa (D)\geq 1$. 
Then, we can conclude that $D$ is semiample
from Kawamata's basepoint-free theorem (see \cites{fujino-ok, kawamata}).  
\end{proof}

Note that, in the above proof, 
we have not used the fact that $D$ has at worst canonical singularities 
and  $H^1(D, \mathcal{O}_D)=0$. 
Proposition \ref{prop.cy} shows that 
a prime Calabi--Yau divisor $D$ can be embedded into a fibration $\pi\colon X\to C$ as a fiber, 
which allows us to consider Serrano's conjecture for a family of Calabi--Yau varieties rather than for
the single Calabi--Yau  variety $D$. 
This observation will be discussed in detail in Corollary \ref{cor.abundance.eff}.

\section{Abundance for strictly nef divisors}\label{sec3}

In this section,  following the inductive strategy of \cite{serrano}, 
we prove that strictly nef divisors on K-trivial fourfolds are ample 
under the assumption of $\mathbb Q$-effectivity. 

For this purpose, 
we first solve Conjecture \ref{big-conj} in the two-dimensional case.
For a log canonical pair $(S,\Delta)$ of dimension two,  
Conjecture \ref{big-conj}  has been proved 
in the case where $S$ is smooth, $\Delta=0$, and $\kappa(S)=0$ 
(see \cite{ccp}*{Proposition 2.5}) and 
in the case where $\kappa(S)\geq 1$ (see \cite{ccp}*{Remark 2.6}). 
Our result generalizes the above results in full generality.

\begin{thm}\label{surf-big}
Let $(S,\Delta)$ be a log canonical pair of dimension two
and $L$ be an almost strictly nef Cartier divisor on $S$.
Then, the $\mathbb Q$-Cartier divisor $K_S+\Delta+tL$ is big for $t>3$.
\end{thm}

\begin{proof}
Let $\pi\colon \bar{S}\to S$ be the relative minimal resolution, 
and $\mu\colon S\to S^*$ be a surjective birational morphism
with a strictly nef $\mathbb Q$-Cartier divisor $L^*$ on $S^*$ such that 
$L=\mu^*L^*$. 
By Remark \ref{rem.asn}, 
we may assume that $S^*$ is normal and  $\mu$ has connected fibers.
Then, running a relative $K_{\bar S}$-MMP over $S^*$,
we obtain
\[
\xymatrix{\bar S\ar[r]^{\tau}& R\ar[r]^\rho & S^*
},
\] 
where $R$ is the relative minimal resolution of $S^*$. 
In particular, the canonical divisor $K_R$ is $\rho$-nef.
We have $K_{\bar S}+t\pi^*L\leq \pi^*(K_S+\Delta+tL)$ since $\pi\colon \bar{S}\to S$ is a relative minimal resolution. 
Hence, it is sufficient to show that $K_{\bar S}+t\pi^*L=K_{\bar S}+t\pi^*\mu^*L^*$ is big for $t>3$.
This is equivalent to showing that $K_R+t\rho^*L^*$ is big for $t>3$ 
since $\pi^*\mu^*L^*=\tau^*\rho^*L^*$ is $\tau$-trivial. 
Note that $K_R+t\rho^*L^*$ is nef for $t\geq 3$ by Lemma \ref{lem.sm.nef}.

For a contradiction, we assume that $K_R+t\rho^*L^*$ is not big for some $t>3$. 
Then, since we have 
\begin{align*}
0&=(K_R+t\rho^*L^*)^2 \\
&=(K_R+3\rho^*L^*+(t-3)\rho^*L^*)^2 \\
&=(K_R+3\rho^*L^*)^2+2(t-3)(K_R+3\rho^*L^*) \cdot \rho^*L^*+(t-3)^2(\rho^*L^*)^2, 
\end{align*}
we can obtain that  
$(K_R+3\rho^*L^*)^2=(K_R+3\rho^*L^*) \cdot \rho^*L^*=(\rho^*L^*)^2=0$. 
This indicates that 
\begin{equation}\label{eq3.1}
K_R^2=K_R\cdot \rho^*L^*=(\rho^*L^*)^2=0.
\end{equation}

We first consider the complete linear system $|rK_R|$ for a positive integer $r$. 
If there exists a nonzero divisor $D \in |rK_R|$ for some $r>0$, 
then $D^2=r^2 K_R^2=0$ from \eqref{eq3.1}, 
and thus 
$D$ is not contained in the $\rho$-exceptional locus 
from Grauert's criterion (see \cite{bpv}*{\S 3, Theorem 2.1}). 
This contradicts  the strict nefness of  $L^*$. 
Indeed, we have 
$$
0< \rho(D) \cdot  L^*=D \cdot \rho^*L^* =rK_R\cdot \rho^*L^*=0
$$
by \eqref{eq3.1}.  
Therefore, we can conclude that 
$|rK_R|= \emptyset$ for any $r>0$ or $rK_R \sim 0$ for some $r>0$.

By the same argument applied to $|r\rho^*L^*|$, we can conclude that 
$|r\rho^*L^*|= \emptyset$ for any $r>0$ or $r\rho^*L^* \sim 0$ for some $r>0$. 
The latter contradicts the strict nefness of $L^*$; 
hence $|r\rho^*L^*|= \emptyset$ holds for any $r>0$. 
By the same argument applied to $|K_R-r\rho^*L^*|$, 
we can conclude that $|K_R-r\rho^*L^*|= \emptyset$ for any $r>0$ or 
$K_R-r\rho^*L^* \sim 0$ for some $r>0$. 
Note that $|K_R-(r+m)\rho^*L^*| =|-m\rho^*L^*| = \emptyset$ holds for any $m>0$ in the latter case. 
Then, by using the Serre duality, the Hirzebruch--Riemann--Roch formula,  and \eqref{eq3.1}, we can show that 
\[
0\geq -h^{1}(R, r\rho^*L^*)=\chi(R, r\rho^*L^*)=\chi(\strutt_R)
\]
for $r\gg 1$, which implies that $q(R)\geq 1$.

If $K_{R}+t\rho^*L^*$ is strictly nef, 
the solution of Serrano's conjecture for surfaces \cite{serrano}
shows that 
\[
(k+1)K_R+tk\rho^*L^*=K_R+k(K_{R}+t\rho^*L^*)
\] 
is ample for $k>3$. 
This is a contradiction to \eqref{eq3.1}. 
Hence, we may assume that there exists a curve $C $ on $R$ 
such that $(K_{R}+t\rho^*L^*)\cdot C=0$ for some $t>3$. 
Then, we obtain that 
\[
0=(K_{R}+t\rho^*L^*)\cdot C = (K_{R}+3\rho^*L^*)\cdot C + (t-3)\rho^*L^*\cdot C
\geq (t-3)\rho^*L^*\cdot C\geq 0,
\]
since $K_R+3\rho^*L^*$ is nef by Lemma \ref{lem.sm.nef} and 
$\rho^*L^*$ is also nef by definition. 
This indicates that 
\begin{align*}
\rho^*L^*\cdot C=0 \text{ and } K_R\cdot C=(K_{R}+t\rho^*L^*)\cdot C=0. 
\end{align*}
Then, since $L^*$ is strictly nef, 
the curve $C$ is $\rho$-exceptional; 
hence we have $C^2<0$. 
Moreover, we can see that $C$ is a $(-2)$-curve from $K_R\cdot C=0$ and the adjunction formula.

We now derive a contradiction case by case.
Since $|rK_R|= \emptyset$ for any $r>0$ or $rK_R \sim 0$ for some $r>0$,
we have that either $\kappa(R)=-\infty$ or $\kappa(R)=0$ holds. 
Let us now consider the  case $\kappa(R)=-\infty$.  
Then, it follows from $K_R^2=0$ and $q(R)\geq 1$  that 
$R$ is a minimal geometrically ruled surface over a curve of genus $\geq 1$.
This surface contains no $(-2)$-curve.
We finally consider the case $\kappa(R)=0$. 
Then, it follows that $R$ is minimal from $K_R^2=0$, 
and thus $R$ is either bi-elliptic or abelian.
In particular, the surface $R$ contains no rational curves. 
\end{proof}

Conjecture \ref{serrano.log.version} in the two-dimensional case 
immediately follows from Theorem \ref{surf-big}. 

\begin{cor}\label{surf.ample}
Let $(S,\Delta)$ be a log canonical pair of dimension two
and $L$ be a strictly nef Cartier divisor on $S$.
Then, the $\mathbb Q$-Cartier divisor $ K_S+\Delta+tL$ is ample for $t>4$.
\end{cor}
\begin{proof}
Since $L$ is strictly nef and 
$K_S+\Delta+tL$ is nef for $t\geq 4$ by Lemma \ref{nef-lem}, 
the $\mathbb Q$-Cartier divisor $K_S+\Delta+tL$ is strictly nef for $t>4$.
Theorem \ref{surf-big} indicates that 
$$
2(K_S+\Delta+tL)-(K_S+\Delta)=K_S+\Delta+2tL
$$ is nef and big for $t>4$.
Then, from the basepoint-free theorem for log canonical pairs, 
we can see that $K_S+\Delta+tL$ is semiample, 
and thus $K_S+\Delta+tL$ is ample for $t>4$.
\end{proof}

We are now ready to prove the main result of this section. 

\begin{thm}\label{thm.abundance}
Let $X$ be a K-trivial fourfold and $L$ be a strictly nef 
$\mathbb Q$-Cartier divisor on $X$.
If we have $\kappa(L)\geq 0$, then $L$ is ample.
\end{thm}
\begin{proof}
We will derive a contradiction by assuming that $L^4=0$.
By the assumption $\kappa(L)\geq 0$, 
there exists a positive integer $m$ and an effective divisor $E$ on $X$
such that $E\sim mL$. 
For simplicity, we assume that $E$ is connected.
Let $E=\sum_{i=1}^{n} a_iE_i$ be the irreducible decomposition of $E$. 
Then, we have 
\[
0=mL^4=L^3 \cdot E=\sum_{i=1}^{n} a_i L^3 \cdot E_i, 
\]
and thus we obtain $L^3 \cdot E_i=0$ for any $i$. 
Set $V:=E_1$ and $F:=\sum^n_{i\geq 2}a_iE_i$.
Note that $F$ could be the zero divisor.
Let us consider the morphisms
\[
\nu\circ \pi =\mu \colon T \xrightarrow{\quad \pi \quad}  V^{\nu}  \xrightarrow{\quad \nu \quad}  V, 
\] 
where $\mu:=\nu\circ \pi$, $\nu\colon V^{\nu}\to V$ is the normalization, 
and $\pi\colon T\to V^{\nu}$ is a relative canonical model.
By subadjunction (see \cite{jiang}*{Proposition 5.1}), 
we can find an effective divisor $\Delta_{V^{\nu}}$ on $V^{\nu}$ such that 
\[
\nu^*V|_{V^{\nu}}=\nu^*(K_X+V)|_{V^{\nu}}=K_{V^{\nu}}+\Delta_{V^{\nu}}. 
\]

Since $T$ is a relative canonical model,
there exists an effective divisor $G$ on $T$ such that
$K_T+G=\pi^*K_{V^{\nu}}$. Set $B:=G+\pi^*\Delta_{V^{\nu}}+(1/a_1)\mu^*F$ and $D:=m\mu^*L|_T$.
Then, we can see that 
\begin{equation}\label{eq.equal}
K_T+B=\mu^*(V+\frac{1}{a_1}F)|_T=\frac{m}{a_1}\mu^*L|_T=\frac{1}{a_1}D
\end{equation}
is almost strictly nef.
If $K_T+tD$ is big for some $t>0$, then  
$K_T+B+tD=(1/a_1+t)D$ is big since $B$ is effective; 
hence $L|_V$ is also big, 
which is a contradiction to $L^3 \cdot V=L^3 \cdot E_1=0$.
Therefore, we may assume that $K_T+tD$ is not big for any $t>0$. 

As $K_T$ is $\mu$-ample and $D=\mu^*(mL|_V)$ is Cartier,  
it follows that $K_T+tD$ is nef for $t\geq 6$ from Lemma \ref{nef-lem}.
Therefore, we obtain $(K_T+tD)^3=0$ for $t\geq 6$ since $K_T+tD$ is not big, 
which yields that 
\begin{align*}
0&=(K_T+6D+(t-6)D)^3 \\
&=(K_T+6D)^3+3(t-6)(K_T+6D)^2 \cdot D+3(t-6)^2(K_T+6D)\cdot D^2+(t-6)^3D^3. 
\end{align*}
All the terms on the right-hand side should be zero since $D$ is also nef. 
Hence, we can conclude that 
\[
K_T^3=K_T^2 \cdot D=K_T \cdot D^2=D^3=0.
\]
Combining this with \eqref{eq.equal}, we can deduce that  
\begin{equation}\label{eq.key}
\begin{split}
&D^2 \cdot B=D^2 \cdot  (K_T+B)=\frac{1}{a_1}D^3=0,
\\ 
&D  \cdot K_T \cdot  B=D \cdot  K_T \cdot  (K_T+B)=\frac{1}{a_1}K_T\cdot  D^2=0,
\end{split}
\end{equation}
and
\begin{equation}\label{eq.key.2}
\begin{split}
&K^2_T \cdot  B=K_T^2 \cdot  (K_T+B)=\frac{1}{a_1}K_T^2 \cdot  D=0,
\\ 
&K_T \cdot  B^2=K_T \cdot  (K_T+B)^2=\frac{1}{a^2_1}K_T \cdot  D^2=0,
\\ 
&B^3=(K_T+B)^3=\frac{1}{a^3_1}D^3=0.
\end{split}
\end{equation}
We now derive contradictions on (some components of) $B$ using these equations.

\begin{case}
Let us 
consider the case where there exists an integral component $S$ of $B$ 
such that $S$ is not $\mu$-exceptional. 
As $D$ and $K_T+tD$ are nef, we have that $0=D^2 \cdot  B\geq a_S D^2 \cdot  S\geq 0$ and
\[
0=(K_T+tD) \cdot D \cdot B\geq a_S(K_T+tD) \cdot D  \cdot S\geq 0
\] 
by \eqref{eq.key}, where $a_S> 0$ is the coefficient of $S$ in $B$. 
This indicates that 
\begin{equation}\label{eq.key.1}
D \cdot K_T \cdot S=D^2 \cdot S=0.
\end{equation}
Let us consider the morphisms 
\[
\xymatrix{R\ar[r]^{\gamma}&  S^{\nu} \ar[r]^{\nu}& S\ar[r]^\mu & S^*}, 
\]
where $S^*:=\mu(S)$. 
Note that $S^*$ is a surface on $V$ 
as $S$ is not contained in the $\mu$-exceptional locus.
Note also that $\nu\colon S^{\nu}\to S$ is the normalization and
$\gamma\colon R\to S^{\nu}$ is a relative minimal resolution.
We can see that 
$$
\gamma^*\nu^*D|_R
=m\gamma^*\nu^*\mu^*L|_R=m\gamma^*\nu^*\mu^*(L|_{S^*})
$$
is an almost strictly nef Cartier divisor on $R$.
It follows that $K_R+r\gamma^*\nu^*D|_R$ is big for $r>3$ from  Theorem \ref{surf-big}.
By subadjunction (see \cite{jiang}*{Proposition 5.1}), 
we have 
$K_R\leq \gamma^*\nu^*(K_T+S)|_R$, and so 
$(K_T+S+rD)|_S$  is also big for $r>3$.
Moreover, by the Zariski decomposition (see \cite{fujino-foundations}*{Theorem 3.5.2}), 
we can see that 
a curve $C\in |K_T+S+rD|$ is not contained in the $\mu|_S$-exceptional locus. 
Then, from \eqref{eq.key} and \eqref{eq.key.1}, we can conclude that
\[
0<D|_S\cdot C=kD \cdot (K_T+S+rD)  \cdot S=kD \cdot K_T  \cdot S+kD \cdot S^2
+krD^2 \cdot S=kD \cdot S^2.
\]
On the other hand, we have that 
\[
0=D \cdot (K_T+B)  \cdot S=D \cdot K_T \cdot S+D\cdot B\cdot S=
D\cdot (B-a_SS) \cdot S+a_S D \cdot S^2\geq a_SD \cdot S^2. 
\]
This is a contradiction.
\end{case}
\begin{case}
Let us consider the case where any integral components of $B$ are $\mu$-exceptional.
In this case, both $F$ and $\Delta_{V^{\nu}}$ are the zero divisors
(i.e., $E=a_1V$, $\mu=\pi$, and $V$ is normal).
As $T$ is a relative canonical model, the canonical divisor $K_T$ is $\mu$-ample. 
For a contradiction, we assume that $B\neq 0$.  
We have that $\dim \mu(B)\ge 1$ by $B^3=0$ (see \eqref{eq.key.2}). 
Take a very ample divisor $H$ on $V$ such that $K_T+\mu^*H$ is ample on $T$.
Then, for a sufficiently large integer $s$,  
we can see that $s(K_T+\mu^*H)\cdot B$ is a curve not contracted by $\mu$.
The divisor $\mu_*D=mL|_V$ is strictly nef, and so 
$s(K_T+\mu^*H) \cdot B \cdot D>0$ by the projection formula. 
Combining this with \eqref{eq.key},
it follows that $H \cdot \mu(B) \cdot(mL|_V)=\mu^*H \cdot B \cdot D>0$, which is a contradiction.
Therefore, it turns out that $B=0$ and $T=V$. 
However, this means that $K_T=V|_V\sim_{\mathbb Q}L|_V$ is strictly nef.
By the abundance theorem for threefolds, we can see that 
$K_T$ is semiample, and hence ample.
This contradicts $K^3_T=0$.
\end{case}
In conclusion, we can obtain $L^4>0$ in any case. 
Then, by the basepoint-free theorem for K-trivial fourfolds, 
the divisor $L$ is semiample, and thus ample.
\end{proof}

\begin{rem}
Most of the discussions in Case 2 were contributed by Chen Jiang. 
The whole proof works almost verbatim in a more general case; 
specifically, when $(X, \Delta)$ is a log canonical pair of dimension four such that  $K_X+\Delta$ is strictly nef, 
we can prove that $K_X+\Delta$ is ample  if $\kappa(K_X+\Delta)\geq 0$.
This is discussed in an unpublished paper by Chen Jiang and the first author.
\end{rem}

If Conjecture \ref{serrano.log.version} holds true in  dimension three
and $\kappa(K_X+\Delta+tL)\geq 0$ for some $t\in \mathbb Q$, 
then we can even prove Conjecture \ref{serrano.log.version} in dimension four using very similar arguments.
We state the precise result for K-trivial fourfolds and sketch the proof below to see how it works.

\begin{thm}\label{thm.abundance.eff}
Let $X$ be a K-trivial fourfold, $V$ be a prime divisor, and $L$ be a strictly nef $\mathbb Q$-Cartier divisor
on $X$. If $V$ is not a prime Calabi--Yau divisor,
then $V+tL$ is ample for $t\gg 1$.
\end{thm}
\begin{proof}
Let $\delta\in \mathbb Q_{>0}$ be a small number such that $(X, \delta V)$ is a log canonical pair.
We may assume that $L$ is Cartier. Then,
by Lemma \ref{nef-lem}, the $\mathbb Q$-divisor $\delta V+sL$ is strictly nef for $s>8$. 
In particular, $V+tL$ is strictly nef for $t>\lceil \frac{8}{\delta}\rceil$.
Assume that $V+tL$ is not ample. 
Then $V+tL$ is not big for any $t>\lceil \frac{8}{\delta}\rceil$.
Simple calculations (as in Theorems \ref{surf-big} and \ref{thm.abundance}) show that
\begin{equation}\label{eq.key.3}
0=V^3 \cdot L=V^2 \cdot L^2=V \cdot L^3=L^4.
\end{equation}
Let us consider the morphisms 
\[
\nu\circ \pi =\mu \colon T \xrightarrow{\quad \pi \quad}  V^{\nu}  \xrightarrow{\quad \nu \quad}  V, 
\] 
where $\mu:=\nu\circ \pi$, 
$\nu\colon V^{\nu}\to V$ is the normalization, 
and $\pi\colon T\to V^{\nu}$ is a relative canonical model. 
Set $D_t:=\mu^*(V+tL)|_T$,
which is almost strictly nef for $t\gg 1$.
We can choose $t$ to be an integer such that $V+tL$ is Cartier.
Then, $K_T+6D_t$ is nef for $t\gg 1$ by Lemma \ref{nef-lem}.
As in Theorem \ref{thm.abundance},
there exists an effective divisor $B$ such that
$K_T+B=\mu^*V|_T$
and $K_T+B+6D_t=\mu^*(7V+6tL)|_T$ is nef for $t\gg 1$,  but not 
big by \eqref{eq.key.3}.
It follows that $K_T+6D_t$ is nef but not big for $t\gg 1$. 
Combining this with \eqref{eq.key.3}, we can conclude  that 
\[D_t^2 \cdot B= D_t \cdot K_T \cdot B=B^3=0
\]
for $t\gg 1$ using similar calculations as in Theorem \ref{thm.abundance}.

Suppose that some integral component of $B$ is not $\mu$-exceptional. 
Then, the proof of Case 1 of Theorem \ref{thm.abundance}
can be applied without any modification to obtain a contradiction. Therefore, we assume that $\dim \mu(B)\leq 1$.
In this case, $B$ is $\mu$-exceptional and $V$ is normal. 
Assume that $B\neq 0$. Then, $\dim \mu(B)=1$ by $B^3=0$.
As in Theorem \ref{thm.abundance}, $(K_T+\mu^*H) \cdot B \cdot D_t>0$ for a very ample divisor $H$ on $V$. 
It follows that
 $H \cdot \mu(B) \cdot (V+tL)|_V>0$, which is impossible. Therefore, $B=0$, $T=V$, and 
$D_t=K_T+tL|_T$ is not ample for $t\gg 1$.
Note that $T$ is a Gorenstein variety with at worst canonical singularities.

If $q(T)=h^1(T,\mathcal O_T)\neq 0$, 
then we can prove that $K_T+tL|_T$ is ample for $t\gg 1$, 
using the Albanese map (as in \cite{serrano} or \cite{ccp}*{Section 3}).  
Thus, we assume that $h^1(T,\mathcal O_T)=0$.
In this case, we could also obtain the ampleness from the proofs in
\cites{ccp, oguiso, serrano} if $K_T$ is not numerically trivial.
Therefore, we further assume that $K_T\equiv 0$, and thus $K_T\sim_{\mathbb Q} 0$.
If $K_T\not\sim 0$, then we have 
$$h^3(T,\mathcal O_T)=h^0(T,\mathcal O_T(K_T))=0
\text{ and }
\chi(T,\mathcal O_T)=h^0(T,\mathcal O_T)+h^2(T,\mathcal O_T)\geq 1.
$$
Then, we can see that 
\[
\chi(T,mL|_T)=\frac{(L|_T)^3}{6}m^3+\frac{c_2(T) \cdot L|_T}{12}m+\chi(T,\mathcal O_T)\geq 1
\] for $m\geq 0$,
where $c_2(T)$ is the pushforward of the second Chern class of some resolution 
(see \cite{oguiso}*{Lemma 1.4}). Then,  $K_T+tL|_T$ is ample for $t\gg 1$ according to the proof of
\cite{oguiso}*{Proposition 2.7}. Therefore, the remaining case is that 
$K_T\sim 0$ and $H^1(T,\mathcal O_T)=0$, that is, $T=V$ is a prime Calabi--Yau divisor.
\end{proof}

If Conjecture \ref{serrano.log.version} holds in dimension three,
then the proof of Theorem \ref{thm.abundance.eff} can stop at the end of the second paragraph.
As Conjecture \ref{serrano.log.version}  holds in dimension two (see Corollary \ref{surf.ample}),
we have an unconditional result in dimension three using the same technique. 
\begin{thm}\label{thm.ab.dim3}
Let $X$ be a K-trivial threefold, $V$ be a prime divisor, and $L$ be a strictly nef $\mathbb Q$-Cartier divisor on $X$.
Then, $V+tL$ is ample for $t\gg 1$. In particular,
if $\nu(L)=2$, then $L|_V$ is ample on $V$ for any prime divisor $V$. 
\end{thm}
\begin{proof}
We only prove the last statement. By assumption, we have $L^3=0$. 
As $V+tL$ is ample for $t\gg 1$, we can see that $L^2 \cdot (V+tL)>0$ 
by Kleiman's ampleness criterion (see \cite{liu-svaldi}*{Lemma 2.1}). 
 It follows that $(L|_V)^2=L^2\cdot V>0$; 
 hence, $L|_V$ is ample by the Nakai--Moishezon criterion.
\end{proof}

\begin{rem}\label{rem.hk}
It has been pointed out by Chen Jiang that 
$V+tL$ is ample for $t\gg 1$ unless $L\sim_{\mathbb Q} V$ 
for any prime divisor $V$ and any strictly nef divisor $L$ on a hyperk\"ahler manifold.
This follows directly from Lemma \ref{nef-lem} and the Hodge index theorem on  hyperk\"ahler manifolds
(see \cite{verbitsky}*{Lemma 3.13}). 
\end{rem}

From the above results,
it is (not) surprising to see 
that the log version of Serrano's conjecture is easier than its original version.
This is because the effective divisor $V$ allows us to perform induction on the dimension.
This provides evidence supporting the proposed Conjecture \ref{serrano.log.version}.
However, these log version results impose extreme restrictions on the structure of $X$
if $L$ is not ample.
For example, if there exists a strictly nef divisor $L$ on a Calabi--Yau threefold $X$ such that
$L^3=c_2(X) \cdot L=0$, which is the only remaining case of Serrano's conjecture in dimension three,
then 
\begin{equation}\label{eq.example}
c_2(X) \cdot V=c_2(X) \cdot (V+tL)>0
\end{equation} 
for any prime divisor $V$ and $t\gg 1$,
because $c_2(X)\neq 0$ is pseudoeffective by \cite{miyaoka}*{Theorem 6.6} and
 $V+tL$ is ample by Theorem \ref{thm.ab.dim3}.
By taking the limit, we can see that 
$c_2(X) \cdot E\geq 0$ for any pseudoeffective divisor $E$ on $X$.
Equivalently, $c_2(X)$ is a limit of movable curves in the sense of \cite{bdpp} by the duality of cones.
This is a rather strong restriction on the position and shape
 of the pseudoeffective cone in $H^2(X,\mathbb R)$.
Moreover, if $V$ is smooth for simplicity (otherwise, take resolutions; see \cite{oguiso}*{Lemma 3.3}), 
then by the exact sequence
\[
0\to \mathcal{O}_{V}( T_V) \to \mathcal \mathcal{O}_{V}( T_X) \to \mathcal O_V(K_V)\to 0,
\]
we have that $c_2(X) \cdot V=c_2(V)-c_1^2(V)>0$. 
This is also a strong restriction on prime divisors.
For instance, if $V$ is a minimal ruled surface, then $c_2(V)-c_1^2(V)=4(g-1)>0$ implies $g>1$, where $g$ is the genus
of the base curve. 
Conversely, by finding a surface $V$ on $X$ such that $c_2(X)  \cdot V=c_2(V)-c_1^2(V)\leq 0$,
we can show that $L$ has to be ample. 
For example, if there exists a minimal ruled surface $V$ whose base curve is rational or elliptic, 
then $c_2(X)  \cdot V\leq 0$; 
or, if there exists a fibration $\pi\colon X\to Y$ of type $\rm{\,I\,}_0$ or type $\rm{I\hspace{-.01em}I}_0$
in the sense of Oguiso \cite{oguiso}*{Main Theorem},
then $c_2(X)  \cdot \pi^*H=0$ for any ample divisor $H$ on the base $Y$.
As all the above cases contradict \eqref{eq.example}, any strictly nef divisor on them has to be ample.
We will see these restrictions in dimension four in the following section.

\section{Non-vanishing for strictly nef divisors}\label{sec4}

In this section, we study when  a (strictly) nef divisor $L$ on a K-trivial fourfold $X$ 
is $\mathbb Q$-effective (i.e., $\kappa(L)\geq 0$ holds). 
We will examine each case according to the numerical dimension $\nu(L)$. 
The divisor $L$ is automatically big in the case $\nu(L)=4$ and 
is trivial  in the case $\nu(L)=0$ by $h^1(X, \mathcal{O}_X)=0$. 
Hence, for our purpose, it is sufficient to consider $\nu(L)=1, 2, 3$.

\subsection{The case of $\nu(L)=3$}
In the case $\nu(L)=3$, 
using the Kawamata--Viehweg vanishing theorem, 
we can prove the non-vanishing result from the assumption that $L$ is nef.

\begin{thm}\label{thm.num3}
Let $X$ be a K-trivial fourfold and $L$ be a nef 
Cartier divisor with $\nu(L)=3$ on $X$. 
Then, the divisor $L$ is semiample  with $\kappa(L)=\nu(L)=3$.
\end{thm}
\begin{proof}
Let $H$ be a very ample smooth hypersurface on $X$. 
Then, by the assumption $\nu(L)=3$, the restriction $L|_H$ to $H$ is big (and nef). 
Let us consider the long cohomology sequence  induced by the exact sequence:
\[
0\to \mathcal O_X(mL) \to  \mathcal O_X(H+mL)\to \mathcal O_H(K_H+mL)\to 0.
\]
Here we used the adjunction formula $(K_X+H)|_H=K_H$ and $K_X \sim 0$. 
For any $i>0$ and $m\geq 0$, 
we have $H^i(X, \mathcal O_X(H+mL))=0$ by Kodaira's vanishing theorem 
as $H+mL$ is ample for any $m\geq 0$; 
we also have $H^i(H, \mathcal O_H(K_H+mL))=0$ by
the Kawamata--Viehweg vanishing theorem 
as $L|_H$ is nef and big.
These indicate that $H^i(X, \mathcal O_X(mL))=0$ holds for $i\geq 2$ and  $m\geq 0$.
Then, the Hirzebruch--Riemann--Roch formula (see Lemma \ref{lem.fc}) shows that  
$$
h^0(X, \mathcal O_X(mL))=\chi(X, mL)+h^1(X, \mathcal O_X(mL))\geq 2
$$ holds for $m\geq 1$. 
In particular, we can see that $\kappa (L)\geq 1$.
Then, by \cite{fukuda}*{Main Theorem} or \cite{fujino-4fold}*{Theorem 3.3}, 
the divisor $L$ is semiample, and so $\kappa(L)=\nu(L)=3$ by 
Kawamata's basepoint-free theorem (see \cites{fujino-ok, kawamata}).
\end{proof}

\begin{rem}
A  semiample divisor $L$ with $\nu(L)=3$ on K-trivial fourfold  
induces an elliptic fibration $X \to Y$, 
and so $X$ contains a rational curve on some singular fiber (see \cite{dfm}).
Therefore, to find rational curves on a K-trivial fourfold,
we can try to find some smooth rational point
on the quartic hypersurface $Q\subset \mathbb P^{n-1}$
defined by the quartic form on $H^2(X, \mathbb R)$,
where $n=h^2(X, \mathbb R)$ is the second Betti number. 
This is a fundamental question in Diophantine geometry.
\end{rem}

Theorem \ref{thm.num3} implies 
that a strictly nef divisor $L$ with $\nu(L)=3$ is 
automatically ample, 
which enables us to remove an unnecessary assumption in Theorem \ref{thm.abundance.eff} 
(i.e., the assumption  that $V$ is not a prime Calabi--Yau divisor).

\begin{cor}\label{cor.abundance.eff}
Let $X$ be a K-trivial fourfold and $L$ be a strictly nef 
$\mathbb Q$-Cartier divisor on $X$.
Let $V$ be a prime Calabi--Yau divisor on $X$. 
Then, $L|_V$ is ample and $V+tL$ is ample for $t\gg 1$.
\end{cor}
\begin{proof}
By Proposition \ref{prop.cy}, 
the divisor $V$ is semiample and $|kV|$ induces a flat fibration 
 $\pi\colon X\to C$ from $X$ onto a smooth curve $C$ for some 
 positive integer $k$. 
Let $F$ be any fiber of $\pi$.
Assume that $F$ is not a multiple of a prime Calabi--Yau divisor. 
Then, by Theorem \ref{thm.abundance.eff}, $F+tL$ is ample for $t\gg 1$. 
It follows that $L|_F\sim_{\mathbb Q}(F+tL)|_F$ is ample.
In particular, $\nu(L)\geq 3$, and so $L$ is ample by 
Theorem \ref{thm.num3}, and our conclusion automatically follows.

Thus, we assume that $F=rD$ for any fiber, where $D$ is a prime Calabi--Yau divisor and $r\geq 1$. 
If  $\kappa(L|_D)\geq 0$, then the abundance theorem in dimension three 
implies that $L|_D$ is ample.
The argument ends as above. If $\kappa(L|_D)=-\infty$ for any $D$,
then  $H^i(D, mL|_D)=0$ for any  $i\geq 0$ and $m\gg 1$ by a modification of \cite{lop}*{Corollary 3.7}
for canonical singularities. Note that $K_D=D|_D\sim 0$.
By  induction on the number $r$ and the long
cohomology sequence induced by the exact sequence
\[
0\to  \strutt_{D}(mL) \to  
\strutt_{rD}(mL)\to \strutt_{(r-1)D}(mL)\to 0,
\]
we have that $H^i(rD, mL|_{rD})=0$ for any  $i\geq 0$ and $m\gg 1$.
It follows that $R^i\pi_*\mathcal O_X(mL)=0$ and
$H^i(X, \mathcal O_X(mL))=H^i(C, \pi_*\mathcal O_X(mL))=0$ for all $i\geq 0$ and $m\gg 1$.
Then, we obtain $\chi(mL)=0$ for $m\gg 1$, which contradicts Lemma \ref{lem.fc}.
\end{proof}

The proof of Corollary \ref{cor.abundance.eff} also works when 
the Albanese map $\alpha\colon X\to A$ has the image of dimension one. 
This observation leads to the following corollary. 

\begin{cor}
Let $X$ be a projective fourfold with $\chi(\mathcal O_X)\neq 0$, and $L$ be a strictly nef $\mathbb Q$-Cartier divisor on $X$. 
Assume that the Albanese map $\alpha\colon X\to A$ has the image of dimension $\geq 1$.
Then, the divisor $K_X+tL$ is ample for $t\gg 1$.
\end{cor}
\begin{proof}
Taking the Stein factorization of $\alpha\colon X\to A$, 
we may assume that $\alpha$ has connected fibers and the image $\alpha(X)$ is a normal variety.
In the case $\dim \alpha(X)\geq 2$,  the conclusion follows directly from \cite{ccp}*{Corollary 3.6}. 
In the case $\dim \alpha(X)=1$,  
the same argument as in Theorem \ref{thm.abundance.eff} and Corollary \ref{cor.abundance.eff} 
shows that  $\chi(mL)=0$ for $m\gg 1$,
contradicting the assumption that $\chi(\mathcal O_X)\neq 0$.
\end{proof}

\subsection{The case of $\nu(L)\leq 2$}
In the case $\nu(L)=1$, the non-vanishing follows directly from \cite{lp}*{Theorems 8.9, 8.10}.
In this subsection, 
we provide a different approach to the non-vanishing of strictly nef divisors, 
including the case  $\nu(L)=2$.
\begin{lem}\label{lem.pseff}
Let $L$ be a strictly nef divisor on a K-trivial fourfold $X$. 
If $\kappa(L)=-\infty$,  
then the divisor $L-D$ is not pseudoeffective for any nonzero effective divisor $D$ on $X$.
\end{lem}
\begin{proof}
Assume that there exists a nonzero effective divisor $D$ such that $E:=L-D$ is  pseudoeffective. 
By Theorem \ref{thm.abundance.eff} and Corollary \ref{cor.abundance.eff}, 
the divisor $H:=D+tL$ is ample for $t\gg 1$. 
Since $E$ is pseudoeffective and $H$ is ample, 
we have that 
\begin{equation*}
\begin{split}
&L^3\cdot E\geq 0, \quad L^2\cdot  H \cdot E\geq 0, \quad  L\cdot  H^2 \cdot  E\geq 0,
\\ 
&L^3 \cdot H\geq 0, \quad  L^2 \cdot  H^2\geq 0, \quad  L\cdot  H^3\geq 0.
\end{split}
\end{equation*}
We now prove that all the above intersection numbers actually are zero, 
by using 
$$
(t+1)L=E+D+tL=E+H \text{ and } L^4=0. 
$$ 
From $0=(t+1) L^4=L^3 \cdot (E+H)\geq 0$, we obtain $L^3 \cdot E=L^3 \cdot H=0$. 
Then, we have $0=(t+1)L^3 \cdot H=L^2\cdot(E+H)\cdot H\geq 0$,  
and so we obtain $L^2\cdot H \cdot E=L^2 \cdot H^2=0$.
Similarly, 
from $$
0=(t+1)L^2 \cdot H^2=L\cdot (E+H) \cdot H^2=L\cdot H^2 \cdot E+L\cdot H^3\geq 0,
$$
we finally obtain $L\cdot H^2 \cdot E=L \cdot H^3=0$.
The numerical class $H^3$ can be represented by an effective curve since $H$ is ample, 
and thus $L \cdot H^3=0$ contradicts the strict nefness of $L$. 
\end{proof}

From now on, we will freely use the basic notation of the analytic methods in \cite{demailly-book} 
and interchangeably use the terms ``Cartier divisors'', ``invertible sheaves'', and ``line bundles''.

\begin{lem}\label{lem.partial.amp}
Let $L$ be a nef  divisor on a K-trivial fourfold  $X$.
Assume that $L$ admits a singular Hermitian metric $h$ with semipositive curvature current 
such that the closed subschemes $V_m$ defined by multiplier ideal sheaves  $\mathcal I(h^{\otimes m})$ 
are of dimension $ \leq 1$ for any  $m \gg 1$. 
Then, we have $\kappa(L)\geq 0$. 
\end{lem}
\begin{proof}
The proof is based on the strategy of \cite{lop}*{Section 3}.
For a contradiction, we assume that $\kappa(L)=-\infty$. 
Note that it is possible that  $V_m=\emptyset$.

By the hard Lefschetz theorem \cite{dps}*{Theorem 0.1},
the morphism
\[
H^0(X, \Omega^2_X(mL)\otimes \mathcal I(h^{\otimes m}))
\twoheadrightarrow H^2(X, \strutt_X(mL)\otimes \mathcal I(h^{\otimes m}))
\]
is surjective. 
Furthermore, we may assume that $H^0(X, \Omega^2_X(mL))=0$ for any $m \gg 1$
from \cite{lp}*{Theorem~8.1}. 
Therefore, we can conclude that $H^2(X, \strutt_X(mL)\otimes \mathcal I(h^{\otimes m}))=0$ 
for any $m \gg 1$. 
On the other hand, we can see that
$H^2(V_m, mL)=0$ for any $m \gg 1$ by assumption. 
Then, from the long cohomology sequence induced by the exact sequence
\[
0\to \strutt_X(mL)\otimes \mathcal I(h^{\otimes m}) \to  \strutt_X(mL) \to \strutt_{V_m}(mL)\to 0,
\]
we can see that $H^2(X,  \strutt_X(mL))=0$ for any $m \gg 1$.
This indicates that 
$$h^0(X, \mathcal O_X(mL))=\chi(X, mL)+h^1(X, \mathcal O_X(mL))+h^3(X, \mathcal O_X(mL))$$ 
for any $m \gg 1$. 
The right-hand side is positive from Lemma \ref{lem.fc}, 
contradicting the assumption $\kappa(L)=-\infty$.
\end{proof}

\begin{thm}\label{thm-num=2}
Let $L$ be a strictly nef  divisor on a K-trivial fourfold  $X$  
with $\nu(L)\leq 2$.
Then, we have $\kappa(L) \geq 0$.  
\end{thm}
\begin{proof}
For a contradiction, we assume that $\kappa(L)=-\infty$. 
Let $h$ be a singular Hermitian metric on $\strutt_X(L)$
such that the curvature current $\sqrt{-1}\Theta_{h}(L)$ is semipositive. 
Let us first consider the Siu decomposition of $\sqrt{-1}\Theta_{h}(L)$: 
\[
\sqrt{-1}\Theta_{h}(L) = R + [D]. 
\]
Here $[D]$ is the integration current defined by 
$D=\sum_{D_{i}}{\nu(\sqrt{-1}\Theta_{h}(L), D_i)}D_{i}$, 
where $D_i$ runs through the prime divisors on $X$ and  
$$\nu(\sqrt{-1}\Theta_{h}(L), D_i):= \inf_{x \in D_{i}} \nu(\sqrt{-1}\Theta_{h}(L), x)$$ 
is the Lelong number of $\sqrt{-1}\Theta_{h}(L)$ along $D_{i}$. 
If $[D]$ is nonzero, then
$mL - D_{i}$ is pseudoeffective 
for any irreducible component $D_{i}$ of $D$ and $m \gg 1$. 
This contradicts Lemma \ref{lem.pseff}, and so $[D]=0$. 


We now demonstrate that the Lelong number $\nu(R, S):=\inf_{x \in S}\nu(R, x)$ along $S$ 
is zero for any subvariety $S \subset X$ of codimension two. 
When $R$ has a local weight whose unbounded locus is contained in a Zariski closed subset of codimension $\geq 2$, 
the Bedford--Taylor product $R^{2}$ can be defined as a positive $(2,2)$-current 
representing  the numerical class $L^{2} \in H^{2,2}(X, \mathbb{R})$ (for example, see \cite{dem}*{Lemma 7.4}). 
However, the $(1,1)$-current $R$ does not necessarily  have such a local weight; therefore, 
we apply the approximation theorem \cite{dem}*{Main Theorem 1.1} for $R$. 
Then, for any $c>0$ and $\e>0$, we can take a $(1,1)$-current $R_{c,\e}$ with the following properties:
\begin{itemize}
\item[(a)] $R_{c,\e}$ represents the numerical class $L \in H^{1,1}(X, \mathbb{R})$. 
\item[(b)] $R_{c,\e}$ is smooth on $X \setminus E_{c}(R)$. 
\item[(c)] $R_{c,\e}\geq - (\min{\{\lambda_\e, c\}} + \e) \omega$. 
\item[(d)] $\nu(R_{c,\e},x)=\max{\{\nu(R,x)-c, 0\}}$. 
\end{itemize}
Here $E_{c}(R):=\{x \in X \,|\, \nu(R, x) \geq c\}$,
$\omega$ is a fixed K\"ahler form on $X$,
and $ \lambda_\e$ is a continuous positive-valued function on $X$. 
It follows from $[D]=0$ that $E_{c}(R)$ is a Zariski closed subset of codimension $\geq 2$.
Hence, properties (a) and (b) enable us to define the Bedford--Taylor product 
$R_{c,\e}^2=R_{c,\e}\wedge R_{c,\e}$ representing  the numerical class $L^{2} \in H^{2,2}(X, \mathbb{R})$.

We now consider the Siu decomposition of $R_{c,\e}^2 $: 
\[
R_{c,\e}^2 = T_{c,\e} + [S_{c,\e}], 
\]
where $[S_{c,\e}]$ is the integration current defined by 
$S_{c,\e}=\sum_{S_{i}}{\nu(R_{c,\e}^2, S_i)}S_{i}$. 
From \cite{dem}*{(7.5)} and property (d), we obtain 
\begin{align}\label{lelong}
[S_{c,\e}]=\sum_{S_{i}}{\nu(R_{c,\e}^2, S_i)} [S_{i}]
\geq \sum_{S_{i}}({\max{\{\nu(R,S_i)-c, 0\}}})^2 [S_{i}]. 
\end{align}
The integrations of $T_{c,\e}\wedge \omega^2$  and $S_{c,\e} \wedge \omega^2$ are uniformly bounded; 
hence, by the weak compactness of currents, 
we can choose subsequences weakly converging to some $(2,2)$-currents $T$ and $S$, respectively,  
as $\e \to 0$ and $c \to 0$. 
It follows from property (c) that the weak limits $T$ and $S$ are positive $(2,2)$-currents. 
Furthermore, the Lelong numbers are nondecreasing under the operator of taking weak limits; 
hence, from inequality \eqref{lelong}, we obtain 
\begin{align}\label{lelong2}
L^{2} \ni T + S \geq T  + \sum_{S_{i}}\nu(R,S_i)^{2} [S_{i}]. 
\end{align}

We will deduce that $L|_{S_i}$ is numerically zero for any $S_i$ with $\nu(R,S_i)>0$ from \eqref{lelong2}. 
Taking the wedge products of $\{\omega\}$ and the numerical class $L$ yields 
that 
$$
H^{4,4}(X, \mathbb{R}) \ni 0=L^{3} \cdot \{\omega\} 
\geq  \{T\} \cdot L  \cdot  \{\omega\} + \{\sum_{S_{i}}\nu(R,S_i)^{2} [S_{i}] \}\cdot L  \cdot  \{\omega\},   
$$
where $\{\bullet \}$ denotes the the numerical class of $(1,1)$-currents. 
The equality on the left-hand side follows from the assumption $\nu(L)\leq 2$. 
Then, since $L$ is nef and $T$ is a positive $(2,2)$-current, 
we can conclude that each term on  the right-hand side  is the zero class. 
 In particular, we can see that $\{S\} \cdot L=0$ 
for any subvariety $S$ satisfying $\codim S =2$ and $\nu(R, S) >0$. 
On the other hand, from the strict nefness, 
we have that $C \cdot  L >0$ 
for a curve $C:=S \cap H$, where $H$ is a very ample hypersurface. 
This contradicts $\{S\} \cdot L=0$,
and so the upper-level set of Lelong numbers 
$E_{c}(\sqrt{-1}\Theta_h(L)):=\{x \in X \,|\, \nu(\sqrt{-1}\Theta_h(L), x) \geq c\}$ is a (possibly) 
countable union of Zariski closed subsets of codimension $\geq 3$. 
Then, by Skoda's lemma, we can see that 
$V_m$ in Lemma \ref{lem.partial.amp} is of codimension $\geq 3$, 
and thus we obtain the conclusion  $\kappa(L) \geq 0$. 
\end{proof}

At the end of this paper, apart from strictly nef divisors, 
we will examine how close our method is 
from the following non-vanishing conjecture for hyperk\"ahler manifolds.

\begin{conj}[Non-vanishing conjecture for hyperk\"ahker manifolds]\label{SYZ}
Let $X$ be a $($not necessarily projective$)$ hyperk\"ahker manifold. 
Let $L$ be a nef divisor on $X$ with $q_{X}(L,L)=0$, 
where $q_X(\cdot, \cdot)$ denotes the Bogomolov--Beauville--Fujiki form on $X$. 
Then, we have $\kappa(L) \geq 0$. 
\end{conj}
Applying the same arguments as in the proof of Theorem \ref{thm-num=2} to Conjecture \ref{SYZ}, 
we obtain the following partial result: 

\begin{thm}\label{main2}
Let $X$ be a  $($not necessarily projective$)$ hyperk\"ahker fourfold  
and $L$ be a nef divisor on $X$ with $q_{X}(L,L)=0$.
Then, one of the followings holds:
\begin{itemize}
\item[$(1)$] $\kappa(L) \geq 0$. 
\item[$(2)$] There exists a holomorphic Lagrangian subvariety $S$ 
such that $L|_{S} \equiv 0$. 
\end{itemize}
\end{thm}
\begin{rem}
A general fiber of a holomorphic Lagrangian fibration $X \to Y$ (if it exists)
is a complex torus satisfying condition (2);
conversely, if a subvariety $S$ in condition (2) can be shown to be a torus, 
then we can find a holomorphic Lagrangian fibration $X \to Y$ 
from a solution of Beauville's question (see \cites{cop, glr, hw}), 
which  proves Conjecture \ref{SYZ}. 
\end{rem}

\begin{lem}\label{lem-min}
Let $\alpha \in H^{1, 1}(X, \mathbb R)$ 
be a $(1,1)$-class on a compact K\"ahler manifold $X$ and 
$T$ be a positive $(1,1)$-current with minimal singularities in $\alpha $. 
We write the divisorial part $D$ of the Siu decomposition $T = R + [D]$ of $T$ 
as $D=\sum_{i \in I} \nu_{i} D_{i}$ with distinct prime divisors $D_i$. 
Then, the numerical classes  $\{D_{i}\}_{i \in I} $  are linearly independent in $H^{1, 1}(X, \mathbb R)$. 
In particular, the divisorial part $D$ is a $\mathbb R$-divisor $($i.e., $I$ is a finite set$)$. 
\end{lem}
\begin{proof}
Note that $D$ may be an infinite sum of the prime divisors $D_{i}$. 
Let $J \subset I$ be a finite subset and $\{a_{j}\}_{j \in J}$ be
positive integers. 
Then, the integration current $[\sum_{j \in J} a_{j}D_{j}]$ is a $(1,1)$-current 
with minimal singularities in its numerical class. 
Indeed, if there exists a positive $(1,1)$-current $S$ in its numerical class whose singularities are less singular than those of $[\sum_{j \in J} a_{j}D_{j}]$, 
then the singularities of $mT - [\sum_{j \in J} a_{j}D_{j}] + S$ 
(which is a positive $(1,1)$-current for a sufficiently large $m \gg 1$)
are less singular than those of $mT$, 
which contradicts the minimal singularities of $mT$ in $m\alpha$.

If the numerical classes  $ \{D_{i}\}_{i \in I}$  are linearly dependent, 
then we can find finite subsets $J, K\subset I$ and nonzero positive integers $\{a_{j}\}_{j \in J}$ and  $\{b_{k}\}_{k \in K}$ 
such that  $J \cap K =\emptyset$ and  
$\sum_{j \in J} a_{j} \{D_{j}\} = \sum_{k \in K} b_{k} \{D_{k}\} $ in $H^{1, 1}(X, \mathbb R)$. 
This is a contradiction to  the minimal singularities of $[\sum_{j \in J} a_{j} D_{j}]$. 
\end{proof}

\begin{proof}[Proof of Theorem \ref{main2}]
We may assume that $L$ is not numerically trivial. 
Let $h$ be a singular Hermitian metric on $L$ with minimal singularities. 
Let us  consider the Siu decomposition of $\sqrt{-1}\Theta_{h}(L)$: 
\[
\sqrt{-1}\Theta_{h}(L) = R + [D]. 
\]
The divisorial part $D$ is a $\mathbb R$-divisor by Lemma \ref{lem-min}.  
We now show that it is sufficient for the proof to consider the case $D=0$
by using the basic results in \cite{huy, huy-2}. 
With respect to the Bogomolov--Beauville--Fujiki form, 
the pseudoeffective cone is the dual cone of the positive cone 
(i.e., the connected component of $\{\alpha \, |\, q_{X}(\alpha,\alpha)>0\}$ containing the K\"ahler cone).
Hence, if $q_{X}(R,R) > 0$, then $R$ (and thus $L$) is big. 
We may assume that $q_{X}(R,R) \leq 0$. 
In the same way, we may assume that $q_{X}(D,D) \leq 0$. 
On the other hand, by construction, 
the numerical class $\{R\}$ lies in the birational K\"ahler cone 
(i.e., the modified nef cone), 
and thus we have $q_{X}(R,R) = 0$. 
Furthermore, we have $q_{X}(L,D) \geq 0$ since $L$ is nef. 
Therefore, we obtain that 
\[
0 = q_X(R,R)=q_X(L,L)-2q_{X}(L,D)+q_{X}(D,D)\leq q_{X}(D,D) \leq 0. 
\]
This yields that $q_X(L,L) = q_X(L,D) =q_X(D,D)=0$. 
Then, it follows that 
$D$ is numerically trivial or $D$ is proportional to $L$ 
from the Hodge index theorem with respect to the Bogomolov--Beauville--Fujiki form. 
In the former case, we obtain $D=0$. 
In the latter case, 
$L$ is numerically equivalent to $D$; 
hence, we have that $\kappa(L) \geq 0$. 

We now consider the case $D=0$. 
If there is no subvariety $S$ on $X$ of codimension two such that  
the Lelong number $\nu(\sqrt{-1}\Theta_{h}(L), S)$ along $S$ is positive, 
then we have $\kappa (L) \geq 0$ by Lemma \ref{lem.partial.amp}. 
Otherwise, for such a subvariety $S$, 
by the same arguments as in the proof of Theorem \ref{thm-num=2}, 
we can conclude that $\{ S \} \cdot L=0$. 
Furthermore, from inequality \eqref{lelong2}, 
we find that 
\[
0 =
\int_X c_{1}(L)^2 \wedge \sigma \wedge \overline{\sigma}
\geq\nu(L,S)^{2}\int_X [S]\wedge \sigma \wedge \overline{\sigma}
\geq 0,
\]
where $\sigma$ is the non-degenerate holomorphic 2-form on $X$.
Here, the equality on the left-hand side follows from $q_{X}(L, L)=0$. 
This indicates that $\sigma|_{S_{\reg}} =0$. 
\end{proof}

Based on the arguments in \cite{cop}, 
we can reduce Conjecture \ref{SYZ} to the case of hyperk\"ahker manifolds of algebraic dimension zero.
\begin{thm}\label{main1}
Conjecture \ref{SYZ} holds 
if Conjecture \ref{SYZ} holds for hyperk\"ahker manifolds of algebraic dimension zero. 
\end{thm}
\begin{proof}
Let us first consider  the case $0 < a(X) < 2n:=\dim X$. 
Let $\phi: X \dashrightarrow Y$ be an algebraic reduction map 
and $\pi: \bar X \to X$ be an elimination of its indeterminacy 
with a morphism $\bar \phi: \bar X \to Y$. 
Then,  the divisor $M$ defined by the reflexive hull of $\pi_{*} \bar \phi^{*} A$ 
satisfies $\kappa(M)=a(X)$ and $q_{X}(M, M)=0$,
 where $A$ is an ample divisor on $Y$. 
As $X$ is non-projective, 
the divisor $M$ with $q_{X}(M, M)=0$ is uniquely determined.
It follows that $L$ is $\mathbb Q$-linearly equivalent to $M$, 
which completes the proof.

We now consider the case in which $X$ is projective. 
In the case $h^{1, 1}(X, \mathbb R) = 1$, the divisor $L$ is ample. 
Therefore, we may assume that $h^{1, 1}(X, \mathbb R) \geq 2$. 
From $H^{0}(X, T_X) \cong H^{0}(X, \Omega_X) = 0$ and $K_X \sim 0$, 
the universal deformation space $\Def{X}$ (resp. $\Def{X, L}$) of $X$ (resp. $(X,L)$) is a complex manifold.
Furthermore,  the tangent space of $\Def{X}$ at $X$ is naturally isomorphic to $H^{1}(X,T_X)$, 
and the tangent space of  $\Def{X, L}$  at $ (X,L)$ is naturally isomorphic to
the kernel of the contraction map
\[
c_{1}(L)\, :\, H^{1}(X, T_X) \to H^{2}(X, \mathcal{O}_X) \cong \mathbb{C} \bar{\sigma}. 
\]
Note that 
$\Def{X, L}$ transversally intersects with $\Def{X, A}$ for any ample divisor $A$,
because the numerical classes of $A$ and $L$ are linearly independent (see, for example, \cite{huy}*{1.12, 1.14}). 
This indicates that $\Def{X, L}$ is not contained in $\Def{X, A}$ for any ample divisor $A$ on $X$. 
Hence, from $h^{1,1}(X) \geq 2$,
we can take a local smooth curve $\Delta$ in $\Def{X, L}$ containing $[X]$ such that 
$\Delta$ is not contained in  $\Def{X, A}$ for any ample divisor $A$ on $X$. 
By construction, there exists some $s \in \Delta$ such that $\bX_{s}$ is non-projective; 
hence, the subset $\{t \in \Delta \,| a(\bX_t) > a(\bX_s) \}$ is a countable union of proper analytic subsets 
by the weak semi-continuity of algebraic dimension (for example, see  \cite{fp}*{Proposition 3.2}). 
We can take a point $t \in \Delta$ such that $\bX_t$ is non-projective and 
\[
\text{
$h^{0}(X,mL) \geq h^{0}(\bX_t, m\bL_t) $ for any $m \in \mathbb{Z}$
}
\]
holds. 
In the case $a(\bX_t) > 0$, we obtain $\kappa(L) \geq \kappa(\bL_t)=a(\bX_t)$ 
from the first half argument. 
In the case $a(\bX_t) = 0$, we have $\kappa(L) \geq \kappa(\bL_t) \geq 0 $ 
from the assumption. 
\end{proof}


\begin{thebibliography}{n} 
\bibitem{bpv}
W.~Barth, K.~Hulek, C.~A.~M.~Peters, A.~Van de Ven, 
\emph{Compact complex surfaces}, 
Second edition. Results in Mathematics and Related Areas. 3rd Series. 
A Series of Modern Surveys in Mathematics, 4. Springer-Verlag, Berlin, 2004.

\bibitem{beauville}
A.~Beauville, 
\emph{Vari\'{e}t\'{e}s K\"{a}hleriennes dont la premi\`ere classe de Chern est nulle}, 
J. Differential Geom. \textbf{18} (1983), no. 4, 755--782. 

\bibitem{bchm}
C.~Birkar, P.~Cascini, C.~D.~Hacon, J.~M\textsuperscript{c}Kernan, 
\emph{Existence of minimal models for varieties of log general type}, 
J. Amer. Math. Soc. \textbf{23} (2010), no. 2, 405--468. 

\bibitem{bl}
Ch.~Birkenhake, H.~Lange, 
\emph{Complex abelian varieties}, 
Grundlehren der Mathematischen Wissenschaften, vol. 302, Springer-Verlag, Berlin, 2004.

\bibitem{bdpp}
S.~Boucksom, J.-P.~Demailly, M.~P\v aun, Th.~Peternell, 
\emph{The pseudo-effective cone of a compact K\"{a}hler manifold and varieties of negative Kodaira dimension}, 
J. Algebraic Geom. \textbf{22} (2013), no. 2, 201--248.

\bibitem{ccp} 
F.~Campana, J.~A.~Chen, Th.~Peternell, 
\emph{Strictly nef divisors}, 
Math. Ann.  \textbf{342} (2008), no. 3, 565--585.

\bibitem{cop}
F.~Campana, K.~Oguiso, Th.~Peternell, 
\emph{Non-algebraic hyperk\"ahler manifolds}, 
J. Differential Geom. \textbf{85} (2010), no. 3, 397--424.

\bibitem{dem}
J.-P.~Demailly, 
\emph{Regularization of closed positive currents and intersection theory}, 
J. Algebraic Geom. \textbf{1} (1992), no. 3, 361--409.
 
\bibitem{demailly-book}
J.-P.~Demailly, 
\emph{Analytic Methods in Algebraic Geometry}, 
Surveys of Modern Mathematics \textbf{1}, International Press, Somerville, MA;  Higher Education Press, Beijing, 2012.

\bibitem{dps}
J.-P.~Demailly, Th.~Peternell, M.~Schneider, 
\emph{Pseudo-effective line bundles on compact k\" ahler manifolds}, 
Int. J. Math. \textbf{12} (2001), no. 6, 689--741.

\bibitem{dfm}
S.~Diverio, C.~Fontanari, and D.~Martinelli, 
\emph{Rational curves on fibered Calabi--Yau manifolds}, 
Doc. Math. \textbf{24} (2019), 663--675.

\bibitem{fp}
A.~Fujiki, M.~Pontecorvo, 
\emph{Non-upper-semicontinuity of algebraic dimension for families of compact complex manifolds}, 
Math. Ann. \textbf{348} (2010), no. 3, 593--599.

\bibitem{fujino-4fold}
O.~Fujino, 
\emph{Finite generation of the log canonical ring in dimension four}, 
Kyoto J. Math. \textbf{50} (2010), no. 4, 671--684.

\bibitem{fujino}
O.~Fujino, 
\emph{Non-vanishing theorem for log canonical pairs}, 
J. Algebraic Geom. \textbf{20} (2011), no. 4, 771--783.

\bibitem{fujino-ok}
O.~Fujino, 
\emph{On Kawamata's theorem}, 
\emph{Classification of algebraic varieties},
EMS Ser. Congr. Rep., Eur. Math. Soc., Z\"urich, 2011, 305--315.

\bibitem{fujino-foundations} 
O.~Fujino, 
\emph{Foundations of the minimal model program}, 
MSJ Memoirs, \textbf{35}. Mathematical Society of Japan, 
Tokyo, 2017.

\bibitem{fukuda} 
S.~Fukuda, 
\emph{On numerically effective log canonical divisors}, 
Int. J. Math. Math. Sci.  \textbf{30} (2002), no. 9, 521--531.

\bibitem{glr} 
D.~Greb, C.~Lehn, S.~Rollenske
\emph{Lagrangian fibrations on hyperk\"ahler manifolds--on a question of Beauville}, 
.Ann. Sci. \'Ec. Norm. Sup\'er. (4) {\bf{46}} (2013), no. 3, 375--403.

\bibitem{hw} 
J.-M.~Hwang, R.M.~Weiss,
\emph{Webs of Lagrangian tori in projective symplectic manifolds},
Invent. Math. {\bf{192}} (2013), no. 1, 83--109.
 
 
\bibitem{han-liu} 
J.~Han, W.~Liu,
\emph{On numerical nonvanishing for generalized log canonical pairs},
Doc. Math. \textbf{25} (2020), 93--123.

\bibitem{hart-1} 
R.~Hartshorne,  
\emph{Ample subvarieties of algebraic varieties}, 
Lect. Notes Math. \textbf{156}, Springer-Verlag, Heidelberg 1970.

\bibitem{huy} 
D.~Huybrechts, 
\emph{Compact hyper-K\" ahler manifolds: basic results}, 
Invent. Math. \textbf{135} (1999), no. 1, 63--113.

\bibitem{huy-2} 
D.~Huybrechts, 
\emph{The K\"ahler cone of a compact hyperk\"ahler manifold}, 
Math. Ann. \textbf{326} (2003), no. 3, 499--513. 


\bibitem{jiang} 
X.~Jiang, 
\emph{On the pluricanonical maps of varieties of intermediate Kodaira dimension}, 
Math. Ann. \textbf{356} (2013), no. 3, 979--1004.

\bibitem{kawamata}
Y.~Kawamata, 
\emph{Pluricanonical systems on minimal algebraic varieties}, 
Invent. Math., \textbf{79} (1985), no. 3, 567--588.

\bibitem{kollar}
J.~Koll\'{a}r, 
\emph{Deformations of elliptic Calabi--Yau manifolds}, 
Recent advances in algebraic geometry, London Math. Soc. Lecture Note Ser., vol. 417,
Cambridge Univ. Press, Cambridge, 2015, 254--290.

\bibitem{kollar-mori}
J.~Koll\'{a}r, S.~Mori, 
\emph{Birational geometry of algebraic varieties},
Cambridge tracts in mathematics, vol. 134, Cambridge University
Press, 1998.

\bibitem{laz} 
R.~Lazarsfeld, 
\emph{Positivity in algebraic geometry I, II}, 
Ergebnisse der Mathematik und ihrer Grenzgebiete, vol. 48, 49, Springer-Verlag, Berlin, 2004.

\bibitem{lop} 
V.~Lazi\'c, K.~Oguiso, Th.~Peternell, 
\emph{Nef line bundles on Calabi--Yau threefolds I}, 
Int. Math. Res. Not. IMRN 2020, no. 19, 6070--6119.

\bibitem{lop1} 
V.~Lazi\'c, K.~Oguiso, Th.~Peternell, 
\emph{The Morrison-Kawamata cone conjecture and abundance on Ricci flat manifolds},
Uniformization, Riemann--Hilbert correspondence, Calabi--Yau manifolds $\&$ Picard--Fuchs equations, 
Adv. Lect. Math. (ALM) 42, Int. Press, Somerville, MA, 2018, 157--185.

\bibitem{lp} 
V.~Lazi\'{c}, Th.~Peternell, 
\emph{Abundance for varieties with many differential forms},
\'{E}pijournal G\'{e}om. Alg\'{e}brique,
\textbf{2} (2018), Art. 1, 35.

\bibitem{lp20} 
V.~Lazi\'{c}, Th.~Peternell, 
\emph{On generalised abundance, I},
Publ. Res. Inst. Math. Sci.,
\textbf{56} (2020), no. 2, 353--389.

\bibitem{liu-svaldi} 
H.~Liu, R.~Svaldi, 
\emph{Rational curves and strictly nef divisors on Calabi--Yau threefolds},
arXiv:2010.12233.

\bibitem{miyaoka} 
Y.~Miyaoka, 
\emph{The Chern Classes and Kodaira Dimension of a Minimal Variety}, 
Adv. Stud. Pure Math. \textbf{10} (1987), 449--476.

\bibitem{okawa}
S.~Okawa, 
{\em{An example of a strictly nef divisor with Kodaira dimension $-\infty$ on a smooth rational surface}},
preprint.

\bibitem{oguiso}
K.~Oguiso, 
\emph{On algebraic fiber space structures on a Calabi--Yau 3-fold}, 
Internat. J. Math. \textbf{4} (1993), no. 3, 439--465, with an appendix by Noboru Nakayama.


\bibitem{serrano} 
F.~Serrano,
\emph{Strictly nef divisors and Fano threefolds}, 
J. Reine Angew. Math. \textbf{464} (1995), 187--206.

\bibitem{verbitsky}
M.~Verbitsky, 
\emph{HyperK\" ahler SYZ conjecture and semipositive line bundles}, 
Geom. Funct. Anal. \textbf{19} (2010), no. 5, 1481--1493.

\bibitem{wilson} 
P.~M.~H.~Wilson, 
\emph{The K\" ahler cone on Calabi--Yau threefolds}, 
Invent. Math. \textbf{107} (1992), no. 3, 561--583.

\bibitem{wilson1} 
P.~M.~H.~Wilson, 
\emph{The existence of elliptic fibre space structures on Calabi--Yau threefolds}, 
Math. Ann. \textbf{300} (1994), no. 4, 693--703.

\end{thebibliography}
\end{document}